\documentclass[12pt,a4paper,final]{article}
\pdfoutput=1
\usepackage{graphicx}
\graphicspath{{./Figures/}}
\usepackage{subfigure}

\usepackage{amssymb,latexsym}
\usepackage{diagbox}

\usepackage{setspace}

\usepackage{amsmath}

\usepackage{mathrsfs,amsfonts}

\usepackage{amsthm,amsxtra}

\usepackage[final]{showkeys}

\usepackage{stmaryrd}

\usepackage{algorithm}
\usepackage{algorithmicx}
\usepackage{algpseudocode}
\usepackage{mathtools}
\newif\ifPDF
\ifx\pdfoutput\undefined
\PDFfalse
\else
\ifnum\pdfoutput > 0
\PDFtrue
\else
\PDFfalse
\fi
\fi

\ifPDF
\usepackage{pdftricks}
\begin{psinputs}
	\usepackage{pstricks}
	\usepackage{pstcol}
	\usepackage{pst-plot}
	\usepackage{pst-tree}
	\usepackage{pst-eps}
	\usepackage{multido}
	\usepackage{pst-node}
	\usepackage{pst-eps}
\end{psinputs}
\else
\usepackage{pstricks}
\fi

\ifPDF
\usepackage[debug,pdftex,colorlinks=true, 
linkcolor=blue, bookmarksopen=false,
plainpages=false,pdfpagelabels,pdfencoding=auto,psdextra]{hyperref}
\else
\usepackage[dvips]{hyperref}
\fi

\usepackage[openbib]{currvita}

\usepackage{fancyhdr}

\newtheorem{theorem}{Theorem}[section]
\newtheorem{lemma}[theorem]{Lemma}
\newtheorem{definition}[theorem]{Definition}

\newtheorem{remark}[theorem]{Remark} 
\newtheorem{corollary}[theorem]{Corollary}

\newtheorem{assumption}[theorem]{Assumption}
\newcommand{\eps}{\varepsilon}

 \newcommand{\bbN}{\mathbb N}
 
\newcommand{\bbR}{\mathbb R} 
\newcommand{\bbZ}{\mathbb Z} 

 \newcommand{\bff}{\mathbf f}

 \newcommand{\bp}{\mathbf p}

\newcommand{\cA}{\mathcal A} \newcommand{\cB}{\mathcal B}
  
\newcommand{\cE}{\mathcal E}

 \newcommand{\cL}{\mathcal L}
 
\newcommand{\cO}{\mathcal O}

\newcommand{\cU}{\mathcal U}

\newenvironment{keywords}
{\noindent{\bf Key words.}\small}{\par\vspace{1ex}}

\makeatletter
\newcommand{\chapterauthor}[1]{%
	{\parindent0pt\vspace*{-25pt}%
		\linespread{1.1}\large\scshape#1%
		\par\nobreak\vspace*{35pt}}
	\@afterheading%
}
\makeatother
\title{How much can one learn from a single solution of a PDE?}
\author{Hongkai Zhao\and Yimin Zhong}
\date{}
\begin{document}
\maketitle
\begin{abstract}
Linear evolution PDE $\partial_t u(x,t) = -\mathcal{L} u$, where $\mathcal{L}$ is a strongly elliptic operator independent of time, is studied as an example to show if one can superpose snapshots of a single (or a finite number of) solution(s) to construct an arbitrary solution. Our study shows that it depends on the growth rate of the eigenvalues, $\mu_n$, of $\mathcal{L}$ in terms of $n$. When the statement is true, a simple data-driven approach for model reduction and approximation of an arbitrary solution of a PDE without knowing the underlying PDE is designed.  Numerical experiments are presented to corroborate our analysis.

\end{abstract}
\section{Introduction}
Partial differential equations (PDE) play an important role in modeling, studying, and predicting complicated dynamics in science, engineering and related fields as well as many real-world applications such as heat convection/diffusion, fluid flow, weather forecast, climate change, etc. A typical procedure is to derive a PDE model, e.g., using physical laws or assumptions, then study the model and develop numerical algorithms for computer simulation. With the advance of technologies, abundant data are available from measurements and observations in many situations where the underlying model is not yet known or not accurate enough, a natural question is whether one can learn or predict the dynamic using data driven approach, which prompt recent increasing interest in PDE learning from its solution data. 

%There are a few proposed approaches for PDE learning. One approach is to learn/train the differential operator as a mapping, e.g., $u(x,t) \rightarrow u(x,t+\Delta t)$, called differential operator approximation (DOA), from solution data \cite{}.  The operator approximation is restricted to a chosen finite dimensional space such as on a discretized meshes in physical domain, or a transformed space, e.g., Fourier/Galerkin space. Since the DOA approach approximates a mapping from an infinite function space to an infinite function space, while general and flexible, it requires large degrees of freedom to represent the mapping, large amount of data, and expensive training process. Another approach is to approximate the underlying PDE using the best
%combination of candidates from a dictionary of basic differential operators and their functions, called differential operator identification (DOI)\cite{}. Since DOI uses terms from a given dictionary (prior knowledge) to approximate the underlying PDE, it can require fewer degrees of freedom, localized solution data and less computational cost \cite{he2022much}. The above two approaches all need to learn the underlying PDE model in some representation from the solution data. 

A basic and essential question for PDE learning is how much solution data is needed and how much data is available realistically. A common issue in most of the previous PDE learning approaches is the assumption of availability of any solution data and as much as one would like. In many real applications, one only has the chance to observe a phenomenon and its dynamics once it happens for some time duration, i.e, a single solution $u(x,t), 0<t \le T$. Even if the event happens again, the environment and setup is different. In our earlier work~\cite{he2022much}, we characterized the solution data space, the dimension of the space spanned by all snapshots in time of a solution  with certain tolerance, and showed how it is affected by the underlying PDE and the initial data for an evolution PDE $\partial_t u(x,t) = -\mathcal{L} u$. In particular, it was shown that the data space is small if $\cL$ is strongly elliptic, which implies limited data from a single solution trajectory and the challenge for PDE learning in practice. It was also shown that the space spanned by all possible solution is small after $t>0$. 

In this work, we study if the space of a single (or a finite number of) solution(s) is as rich as the space of all solutions, or in other words, can one use a superposition of snapshots of a single (or a finite number of) solution(s) to construct an arbitrary solution. When the answer is positive, one can use data driven approach for model reduction and solution approximation bypassing explicit learning or approximation of the PDE operator.
In our study, the following evolution PDE on a compact domain $\Omega \subset \mathbb{R}^d$ satisfying certain homogeneous boundary condition,
\begin{equation}\label{EQ: EVOL}
\begin{aligned}
        \partial_t u(x,t) &= -\mathcal{L} u(x,t),\quad && x\in \Omega \subset \mathbb{R}^d,\;0<t\le T \le \infty,\\ 
        \cB u(x,t)&=0, \quad &&x\in \partial \Omega, \\
        u(x, 0) &= f(x)\quad &&x\in  \Omega,
\end{aligned}
\end{equation}
where $\cL = \sum_{0\le |\alpha|\le N_{\cL}} p_{\alpha}(x) \partial^{\alpha}$ is a  self-adjoint strongly elliptic differential operator of order $N_{\cL}$ with coefficients $p_{\alpha}\in C^{|\alpha| + N_{\cL}/2}(\overline{\Omega})$ and the boundary $\partial\Omega$ is smooth (if not empty) and $\cB u = 0$ denotes the Dirichlet boundary condition $\partial^{\alpha} u (x, t) = 0$ for $x\in\partial\Omega$, $|\alpha|\le \frac{N_{\cL}}{2} - 1$, then $\cL$ is self-adjoint~\cite{berezanskiui1968expansions}.

The rest of the paper is organized as follows. We first present in Section 2 some preliminaries and formulate the problem into a  moment problem. We then consider in Section 3 the cases where eigenvalues $\mu_n$ of $\cL$ are simple and grows super-linearly or (sub-)linearly in terms of $n$, which gives different statements. In Section 4 we study the problem when eigenvalues has finite multiplicities. We then present numerical experiments in Section 5 to verify our analysis and provide an application of data driven approach for solving PDEs without knowing the underlying PDE. We finally give a conclusion in Section 6.
\section{Preliminaries}
 Consider the PDE~\eqref{EQ: EVOL}, for simplicity, we assume the eigenvalues $0 < \mu_{1} < \mu_2 < \dots $ of the self-adjoint strongly elliptic differential operator $\cL$ are positive and distinct and denote $\phi_1(x), \phi_2(x), \dots$ to be the corresponding eigenfunctions which form a complete orthonormal basis in $L^2(\Omega)$. Let $u(x,t)$ be a sample solution which can be represented as
\begin{equation}
    u(x, t) = \sum_{n=1}^{\infty} e^{-\mu_n t} c_n \phi_n(x),  
\end{equation}
where $u(x,0)=\sum_{n=1}^{\infty} c_n \phi_n(x),~ c_n\neq 0$. Given an initial condition $f(x) = \sum_{n=1}^{\infty} f_n \phi_n(x)\in L^2(\Omega)$, we can express the solution at time $\tau$ by 
\begin{equation}
    w(x, \tau) = \sum_{n=1}^{\infty} e^{-\mu_n \tau} f_n \phi_n(x).
\end{equation}
In order to express the solution $w(x,\tau)$ at $\tau>0$ as a superposition of snapshots of the sample solution, we introduce the interpolation or weight function $\nu(t; \tau) \in L^{2}([0, T] )$ and solve the following equation
\begin{equation}
     w(x, \tau)  = \int_0^{T} u(x, t) \nu(t;\tau) dt,
\end{equation}
which is equivalent to solving the exponential moment problem
\begin{equation}\label{EQ: MOMENT}
\int_0^{T} e^{-\mu_n t} \nu(t;\tau) dt = m_n :=   e^{-\mu_n \tau} f_n/ c_n.
\end{equation}

\begin{definition}
The functions $\{\varphi_n\}_{n\ge 1}$ are bi-orthogonal with respect to $\{e^{-\mu_n t}\}_{n\ge 1}$ in $[0,T]$ if 
\begin{equation}
    \int_0^T e^{-\mu_n t} \varphi_k(t) dt = \delta_{kn},
\end{equation}
where $\delta_{kn}$ is the Kronecker delta.
\end{definition}
Suppose the bi-orthogonal functions $\{\varphi_n\}_{n\ge 1}$ exist, then a formal expansion of $\nu(t)$ can be written as
\begin{equation}\label{EQ: NU SERIES}
    \nu(t) = \sum_{n=1}^{\infty} m_n \varphi_n(t).
\end{equation}
However, the exact characterization of the convergence behavior of the series~\eqref{EQ: NU SERIES} is a difficult task. Instead, we will study the absolute convergence in $L^2[0, T]$, which means  
\begin{equation}
    \sum_{n=1}^{\infty} |m_n|\cdot  \|\varphi_n\|_{L^2[0, T]} < \infty.
\end{equation}
If the series indeed converges absolutely, then the series converges in $L^2[0, T]$. The existence of $\nu$ implies $w(x,\tau)$ can be linearly represented by the sample solution $u(x, t)$ on $[0, T]$. Clearly, it only remains to estimate the norms $\|\varphi_n\|_{L^2[0, T]}$. 
\begin{remark}
When there exist non-positive eigenvalues, one can find a constant $\overline{\mu}$ that $\mu_1 + \overline{\mu} > 0$, then the interpolation function can be changed to $\nu(t) e^{\overline{\mu} t}$ instead.  Hence in the rest of this paper, we only focus on the case that all eigenvalues are positive.
\end{remark}
\section{Solution of the moment problem}
\subsection{Convergent series}\label{SEC: CONV}
Let $\cE_T$ be the smallest closed subspace of $L^2[0, T]$ containing the functions $e^{-\mu_n t}$, $n=1, 2, \dots$. It is known~\cite{fattorini1971exact} that $\cE_T$ is a proper subspace of $L^2[0, T]$ if and only if 
\begin{equation}\label{EQ: SERIES}
    \sum_{n=1}^{\infty} \mu_n^{-1} < \infty.
\end{equation}
In the case that~\eqref{EQ: SERIES} is convergent, we define $\cE_T^n$ by the smallest closed subspace of $L^2[0, T]$ containing all $e^{-\mu_k t}$ for $k\neq n$, then $e^{-\mu_n t}\notin \cE_T^n$, there exists a unique element $r_n\in \cE_T^n$ such that minimizes $d_n(T):=\| e^{-\mu_n t}-r_n\|_{L^2[0, T]}$, then the optimal bi-orthogonal function $\varphi_n$ is chosen by
\begin{equation}
    \varphi_n = \frac{e^{-\mu_n t} - r_n}{d_n(T)^2}
\end{equation}
and its norm $\|\varphi_n\|_{L^2[0, T]} = d_n(T)^{-1}$. If $T=\infty$, the computation of $d_n(T)$ is fairly simple~\cite{fattorini1971exact}
\begin{equation}
    d_n(\infty) = \sqrt{\frac{2}{\mu_n}} \frac{\left|\prod_{j\neq n} \left(1 - \frac{\mu_n}{\mu_j}\right)\right|}{\prod_{j\ge 1} \left(1 + \frac{\mu_n}{\mu_j}\right)},
\end{equation}
while for a finite $T$, there exists a constant $\kappa(T) > 0$ such that $d_{n}(T) \ge \kappa(T)^{-1} d_n(\infty)$~\cite{schwartz1943etude}, hence 
\begin{equation}\label{EQ: NORM}
    \|\varphi_n\|_{L^2[0, T]} \le \kappa(T)\sqrt{\frac{\mu_n}{2}}\frac{\prod_{j\ge 1} \left(1 + \frac{\mu_n}{\mu_j}\right)}{\left|\prod_{j\neq n} \left(1 - \frac{\mu_n}{\mu_j}\right)\right|}.
\end{equation}
We estimate $\|\varphi_n\|_{L^2[0, T]}$ under the following two assumptions.
\begin{assumption}\label{AS: 1}
For certain $\sigma\in (0, 1]$ and $\beta > 1$, $\mu_n = M n^{\beta}(1 + o(n^{-\sigma}))$.
\end{assumption}
\begin{assumption}\label{AS: 2}
For certain $\theta > 0$ and $s\ge 0$ such that for each $n\ge 1$, $\mu_{n+1} - \mu_{n} \ge \theta n^{-s}$.
\end{assumption}
Under suitable conditions, the first assumption holds true for the strongly elliptic operators~\cite{zielinski1998asymptotic, safarov1997asymptotic, beals1970asymptotic, agmon1968asymptotic,hormander1968riesz}, where the growth rate is $\beta = N_{\cL}/ d > 1$ when the space dimension is less than the order of the elliptic operator.  The exponent $\sigma$ can be made explicit in certain cases, see Theorem C in~\cite{beals1970asymptotic} and Theorem 3.1 in ~\cite{agmon1968asymptotic}. The second assumption is to prevent the blowing-up of $\|\varphi_n\|_{L^2[0, T]}$. When the eigenvalues are multiple of integers, i.e., $\mu_n=Mn^{\beta}$, this is guaranteed. However in general there is no known estimate for $\mu_{n+1} - \mu_n$ except the spectral gap. For specific cases (see Section~\ref{SEC: EX 3}), it is possible to claim the desired lower bound. Particularly, the one dimensional Sturm-Liouville operators with Dirichlet boundary conditions satisfy both of the assumptions.

In the following, we introduce a lemma for the estimates of the products appearing in~\eqref{EQ: NORM}. The proof is included in the appendix~\ref{AP: LEM} and the main idea follows~\cite{fattorini1971exact}.

\begin{lemma}~\label{LEM: EST}
The following estimates hold.
\begin{equation}\label{EQ: PROD}
    \begin{aligned}
    \prod_{j\ge 1} \left(1 + \frac{\mu_n}{\mu_j}\right) &= \exp \left(  M^{-1/\beta} \mu_n^{1/\beta}  (\zeta_{0, \beta} + o(1)) \right), \\ 
    \left|\prod_{j\neq n}\left(1 - \frac{\mu_n}{\mu_j}\right)\right| &\ge \exp(-K_0 \mu_n^{1/\beta} M^{-1/\beta} (\log \mu_n + K_1) ),
    \end{aligned}
\end{equation}
where $\zeta_{0, \beta}:= \int_0^{\infty} \frac{dy}{y^{1 - 1/\beta}(1+y)}$ and $K_0$, $K_1$ are  positive constants independent of $n$.
\end{lemma}
\begin{lemma}\label{THM: BI-ORTH}
The bi-orthogonal set $\{ \varphi_n \}_{n\ge 1}$ satisfies the following bound
\begin{equation}
    \|\varphi_n\|_{L^2[0, T]} \le \sqrt{ \frac{\mu_n}{2}} \kappa(T) \exp \left( K_0 M^{-1/\beta} \mu_n^{1/\beta} (\log \mu_n + K_2) \right),
\end{equation}
where $K_2$ is a positive constant independent of $n$.
\end{lemma}
\begin{proof}
It is known that
\begin{equation}
\begin{aligned}
     \|\varphi_n\|_{L^2[0, T]} &\le \sqrt{\frac{\mu_n}{2}} \kappa(T) \frac{\prod_{j\ge 1} \left(1 + \frac{\mu_n}{\mu_j}\right)}{\left|\prod_{j\neq n}\left(1 - \frac{\mu_n}{\mu_j}\right) \right|} \\ &\le \sqrt{\frac{\mu_n}{2}}   \kappa(T) \exp \left(  M^{-1/\beta} \mu_n^{1/\beta} \left[ (\zeta_{0, \beta} + o(1)) +K_0 (\log \mu_n + K_1)\right]  \right) \\
     &\le\sqrt{\frac{\mu_n}{2}}  \kappa(T) \exp \left(  K_0M^{-1/\beta} \mu_n^{1/\beta}  (\log \mu_n + K_2) \right),
\end{aligned}
\end{equation}
where the constant $K_2 \ge K_1 + ( \zeta_{0,\beta} + o(1) ) / K_0$.
\end{proof}
Using the above estimate of $\|\varphi_n\|_{L^2[0,T]}$, we immediately obtain the following theorem to characterize the absolute convergence of $\nu$.
\begin{theorem}\label{THM: MAIN}
The exponential moment problem has an absolutely convergent solution in $L^2[0, T]$ if 
\begin{equation}
    \sum_{n=1}^{\infty} |m_n| \sqrt{\mu_n} \exp \left( K_0 M^{-1/\beta} \mu_n^{1/\beta} (\log \mu_n + K_2 ) \right) < \infty
\end{equation}
\end{theorem}

\begin{corollary}\label{COR: 5}
If $\tau > 0$, then the moment problem with $m_n = e^{-\mu_n \tau} \frac{f_n}{c_n}$ has a solution in $L^2[0, T]$ if 
\begin{equation}
    \sum_{n=1}^{\infty} \left|\frac{f_n}{c_n}\right|^2 \exp(-2 \mu_n \tau_0) < \infty.
\end{equation}
for certain $\tau_0 \in [0, \tau)$. 
\end{corollary}
\begin{proof}
From the Theorem~\ref{THM: MAIN}, we just need to show
\begin{equation}
    \sum_{n=1}^{\infty} \frac{f_n}{c_n} \sqrt{\mu_n} \exp\left(-\mu_n \tau + K_0 M^{-1/\beta} \mu_n^{1/\beta} (\log \mu_n + K_2 ) \right) < \infty.
\end{equation}
By the Cauchy-Schwartz inequality,
\begin{equation}
\begin{aligned}
  &\left[\sum_{n=1}^{\infty} \frac{f_n}{c_n} \sqrt{\mu_n} \exp\left(-\mu_n \tau + K_0 M^{-1/\beta} \mu_n^{1/\beta} (\log \mu_n + K_2 ) \right) \right]^2\\
  &\le   \sum_{n=1}^{\infty} \left|\frac{f_n}{c_n}\right|^2 \exp(-2\mu_n \tau_0)\sum_{n=1}^{\infty} \mu_n\exp\left(-2\mu_n (\tau-\tau_0) + 2K_0 M^{-1/\beta} \mu_n^{1/\beta} (\log \mu_n + K_2 ) \right) 
\end{aligned}
\end{equation}
Our conclusion is immediately proved by noticing that when $n$ is sufficiently large,
\begin{equation}
    -2\mu_n (\tau-\tau_0) + 2K_0 M^{-1/\beta} \mu_n^{1/\beta} (\log \mu_n + K_2 ) + \log \mu_n < - \mu_n (\tau- \tau_0).
\end{equation}
\end{proof}

\begin{theorem}\label{THM: MAIN 2}
Suppose $g(x)$ satisfies that
\begin{equation}
    g(x) = \sum_{n=1}^{\infty} c_n \phi_n(x),
\end{equation}
where $|c_n|\ge C e^{-pn^{\alpha}}$ for certain constants $C > 0$, $p \ge 0$ and $\alpha \in [0, \beta)$, $n=1,2,\dots$ and 
 $f\in L^2(\Omega)$.
Let $u(x, t)$ and $w(x,t)$ be the solutions with respect to initial conditions $g(x)$ and $f(x)$ respectively. $\forall \tau> t_0$, there exists an interpolation function $\nu(t; \tau, [t_0, t_1])\in L^2[t_0, t_1]$ such that 
\begin{equation}
    w(x, \tau) = \int_{t_0}^{t_1} u(x, t) \nu(t;\tau) dt,
\end{equation}
and $\|\nu(\cdot\,;\tau)\|_{L^2[t_0, t_1]}\le \frac{1}{\sqrt{2}C} \kappa(t_1 - t_0) \|f\|_{L^2(\Omega)} \cdot  (\tau - t_0)^{-\omega(\tau - t_0)/2}  $ with
\begin{equation}
    \omega(t) = 1 + \varpi t^{-\frac{\max(\alpha, 1)}{(\beta - \max(\alpha, 1) - \delta)}},
\end{equation}
where $\delta\in(0, \beta-\max(\alpha,1))$ is arbitrary and $\varpi$ is a constant independent of $t$.
\end{theorem}
\begin{proof}
By the Corollary~\eqref{COR: 5}, we only need to check if 
\begin{equation}
    \sum_{n=1}^{\infty} \left| \frac{f_n e^{-\mu_n t_0}}{c_n e^{-\mu_n t_0}} \right|^2 e^{-2\mu_n \tau_0} =     \sum_{n=1}^{\infty} \left| \frac{f_n }{c_n} \right|^2 e^{-2\mu_n \tau_0}  <\infty
\end{equation}
for certain $\tau_0\in [0, \tau - t_0)$. Since $|c_n|\ge C e^{-pn^{\alpha}}$, this becomes $   \sum_{n=1}^{\infty} \left| f_n  \right|^2 e^{2pn^{\alpha}} e^{-2\mu_n \tau_0} $, which is finite if $\sum_{n=1}^{\infty} |f_n|^2 <\infty$. Now we give an estimate of $\|\nu(\cdot\,;\tau)\|_{L^2[t_0, t_1]}$. Denoting $T = t_1 - t_0$ and $\tilde{\tau} = \tau - t_0$, we recall~\eqref{EQ: NU SERIES} and Lemma~\ref{THM: BI-ORTH}, 
\begin{equation}\nonumber
\begin{aligned}
    \|\nu(\cdot\,;\tau)\|_{L^2[t_0, t_1]} &\le \sum_{n=1}^{\infty} \left|\frac{f_n}{c_n} \exp(-\mu_n\tilde{\tau} )\right| \sqrt{\frac{\mu_n}{2}}\kappa(T) \exp \left( K_0M^{-1/\beta} \mu_n^{1/\beta}  (\log \mu_n + K_2) \right) \\
    &\le \frac{\kappa(T)}{C}\sum_{n=1}^{\infty} \left|f_n\right|  \sqrt{\frac{\mu_n}{2}} \exp \left(  -\mu_n \tilde{\tau}+ K_0M^{-1/\beta} \mu_n^{1/\beta}  (\log \mu_n + K_2) + pn^{\alpha}\right)\\
    &\le \frac{\kappa(T)}{\sqrt{2}C} \left(\sum_{n=1}^{\infty} |f_n|^2 \right)^{1/2}  \sqrt{ \cA(\tilde{\tau}) },
\end{aligned}
\end{equation}
where $\cA(\tilde{\tau})$ is the summation
\begin{equation}
\begin{aligned}
    \cA(\tilde{\tau}) &=  \sum_{n=1}^{\infty} \exp \left(  -2\mu_n \tilde{\tau}+ 
    \log\mu_n + 2K_0M^{-1/\beta} \mu_n^{1/\beta}  (\log \mu_n + K_2) + 2pn^{\alpha}\right).
\end{aligned}
\end{equation}
We define 
\begin{equation}
\begin{aligned}
      H &:= \sup_{n\ge 1} \frac{ \log\mu_n + 2K_0M^{-1/\beta} \mu_n^{1/\beta}  (\log \mu_n + K_2) }{\mu_n^{1/\beta} \log \mu_n},\\
      N_{\tilde{\tau}} &:= \sup \left\{n\in \bbN\,\Big|\,  H \log \mu_n + 2 p n^{\alpha}\mu_n^{-1/\beta}\ge \tilde{\tau} \mu_n^{1-1/\beta} \right\}, 
\end{aligned}
\end{equation}
then we have the estimate $H = 2K_0 M^{-1/\beta} + O(1)$ and $N_{\tilde{\tau}} \le L \tilde{\tau}^{-1/(\beta-\max(\alpha,1)-\delta)}$ for arbitrary $\delta \in (0, \beta-\max(\alpha,1))$ and a constant $L = L(H, \delta) > 0$ as $\tilde{\tau}\to 0$.  On the other hand, since $xe^{-x} \le e^{-1}$ for any $x > 0$, we can derive
\begin{equation}
    \sum_{n=1}^{\infty} \exp(-\mu_n\tilde{\tau}) \le \frac{1}{e}\sum_{n=1}^{\infty} \frac{1}{\mu_n\tilde{\tau}} = \frac{C'}{\tilde{\tau}},
\end{equation}
where $C' = e^{-1}\sum_{n=1} \mu_n^{-1} < \infty$. Therefore as $\tilde{\tau}\to 0$,
\begin{equation}\nonumber
\begin{aligned}
    \cA(\tilde{\tau}) &\le \sum_{n=1}^{N_{\tilde{\tau}}} \exp \left(  -2\mu_n \tilde{\tau}+ 
    \log\mu_n + 2K_0M^{-1/\beta} \mu_n^{1/\beta}  (\log \mu_n + K_2) + 2pn^{\alpha}\right) \\ 
    &\quad + \sum_{n > N_{\tilde{\tau}}} \exp(-\mu_n\tilde{\tau})\\
    &\le \left[ \sum_{n=1}^{N_{\tilde{\tau}}} \exp(-2\mu_n\tilde{\tau}) \right] \sup_{1\le n\le N_{\tilde{\tau}}} \exp \left(
    \log\mu_n + 2K_0M^{-1/\beta} \mu_n^{1/\beta}  (\log \mu_n + K_2) + 2pn^{\alpha}\right) \\&\quad + \frac{C'}{\tilde{\tau}} \\
    &\le \frac{C'}{\tilde{\tau}} \exp \left(
    H  \beta M^{1/\beta} N_{\tilde{\tau}} \log N_{\tilde{\tau}}(1+o(1)) + 2pN_{\tilde{\tau}}^{\alpha}\right) + \frac{C'}{\tilde{\tau}}\\
    &\le \frac{C'}{\tilde{\tau}}  \left( e^{2p N_{\tilde{\tau}}^{\alpha}+ C'' N_{\tilde{\tau}}\log N_{\tilde{\tau}}}
     + 1 \right),
\end{aligned}
\end{equation}
where $C'' = C''(H, \beta, M)$ is a constant.
\end{proof}
\begin{remark}
It is worthwhile to notice the upper bound for $\|\nu(\cdot\,;\tau)\|_{L^2[t_0, t_1]}$ only depends on the differences $t_1 - t_0$ and $\tau - t_0$ and the initial time $t_0$ does not matter. Actually the dependence on $t_1 - t_0$ from $\kappa(t_1-t_0)$ in~\eqref{EQ: NORM} is quite mild. However, the dependence on $\tilde{\tau}=\tau - t_0$ is significant, which implies stability issue when $\tau$ is close to $t_0$.
\end{remark}
\begin{remark}
 One of the practical issues is that without knowing much of the self-adjoint differential operator $\cL$, can one create a sample solution corresponding to an initial condition satisfying the condition in Theorem~\ref{THM: MAIN 2}? An intuitive choice is to use a point source (or an approximate one in practice) as initial condition, which has the expansion
\begin{equation}
    \delta(x - y) = \sum_{n=1}^{\infty} \phi_n(x) \overline{\phi_n(y)}.
\end{equation}
When the set $ \cap_{n\ge 1}  \{   |\phi_n(z)| > \exp(-p n^{\alpha}) \}$ has a positive measure for certain $p$, then it is possible to randomly select single point sources to fulfill the condition in Theorem~\ref{THM: MAIN 2}. For instance, if we take the domain as $d$-dimension torus $\Omega = \mathbb{T}^d$ and $\cL = -\nabla\cdot (A(x)\nabla )$ that $A(x)$ is Lipschitz, the eigenfunction $\phi_n$ satisfies 
\begin{equation}\label{EQ: EIG EST}
    \sup_{\Omega} |\phi_n| \le C \sup_{E} |\phi_n| \left(C \frac{|\Omega|}{|E|}\right)^{C\sqrt{\mu_n}}
\end{equation}
for certain $C > 0$ for any measurable subset $E\subset \Omega$~\cite{logunov2018quantitative}. If we denote $E = \{|\phi_n|\le e^{-p n^{\alpha}}\}$ and use $ \|\phi_n\|_{L^{\infty}(\Omega)}\ge  \frac{1}{\sqrt{|\Omega|}}$,  then 
\begin{equation}
\begin{aligned}
      |\{|\phi_n(x)|\le e^{-p n^{\alpha}}\}| &\le C|\Omega| \left[\frac{C \exp(-p n^{\alpha})}{1/\sqrt{|\Omega|}}\right]^{\frac{1}{C\sqrt{\mu_n}}} \\&\le C |\Omega| \left[C\sqrt{|\Omega|}\right]^{\frac{1}{C\sqrt{\mu_n}}} \exp(-\frac{p}{C\sqrt{M}}n^{\alpha - \beta/2}(1+o(1))),
\end{aligned}
\end{equation}
where we have used the fact that $\mu_n = M n^{\beta} (1 + o(1))$.
If $\beta > \alpha > \frac{\beta}{2}$, we obtain 
\begin{equation}
    \sum_{n\ge 1}  |\{|\phi_n(x)|\le e^{-p n^{\alpha}}\}| <\infty.
\end{equation}
Therefore as $p\to\infty$, the summation converges to zero. This implies that for sufficiently large $p$, there is a high probability that a random point $z\in\Omega$ that $|\phi_n(z)| > \exp(-pn^{\alpha})$ for all $n\ge 1$. It is unclear whether such statement can be extended to high order strongly elliptic operators.

\end{remark}
\subsection{Divergent series}
It still remains to discuss about the case that
\begin{equation}
    \sum_{n=1}^{\infty} \mu_n^{-1} = \infty, 
\end{equation}
then $\cE_T = L^2[0, T]$~\cite{fattorini1971exact}. For simplicity, we let $T=\infty$,  instead of dealing with the infinite moment sequence, we consider the finite case for the first $N$ moments. Similarly, we define the bi-orthogonal functions $\widetilde{\varphi}_n$, $n=1,2,\dots, N$,
\begin{equation}\label{EQ: FINITE SUBSET}
   \int_0^{\infty} e^{-\mu_n t} \widetilde\varphi_k(t) dt  = \delta_{k, n},\quad 1\le k, n\le N.
\end{equation}
the corresponding solution is
\begin{equation}
    \nu_{N}(t) = \sum_{n=1}^N m_n \widetilde\varphi_n(t).
\end{equation}
where $\widetilde{\varphi}_n$ is optimal in $L^2$ norm. Using the same argument as Section~\ref{SEC: CONV}, the bi-orthogonal functions $\widetilde\varphi_n$ to~\eqref{EQ: FINITE SUBSET} satisfies
\begin{equation}
    \|\widetilde\varphi_n\|_{L^2[0,\infty)} = \sqrt{2{\mu_n}} \frac{\prod_{j\neq n}^{N}(1 + \frac{\mu_n}{\mu_j})}{\left| \prod_{j\neq n}^N (1 - \frac{\mu_n}{\mu_j})\right|},\quad n=1,2,\dots, N.
\end{equation}
For the sake of simplicity, we still make the same assumptions~\ref{AS: 1} and~\ref{AS: 2}. When $0 \le \beta \le 1$ and given any fixed $n$,
\begin{equation}
    \sum_{j\neq n}^{N} \log\left( \frac{1 + \frac{\mu_n}{\mu_j}}{|1 - \frac{\mu_n}{\mu_j}|} \right) = \sum_{j=1}^N  \frac{2\mu_n }{\mu_j} \left(1 + o(1)\right),\text{ as } N\to\infty,
\end{equation}
therefore 
\begin{equation}
    \|\widetilde\varphi_{n}\|_{L^2[0,\infty)} = \sqrt{2\mu_n} \exp\left(2\mu_n \sum_{j=1}^N \mu_j^{-1} \left(1 + o(1)\right)\right),\text{ as } N\to\infty.
\end{equation}
Then $ \|\widetilde\varphi_{n}\|_{L^2[0,\infty)} \to \infty$ as $N\to \infty$, we can conclude the following theorem.
\begin{theorem}\label{THM: NON}
If $\beta \in [0, 1]$, then the series solution $\lim_{N\to\infty}\nu_N$ cannot be absolutely convergent unless $m_n \equiv 0$ for all $n$.
\end{theorem}
However, the solution still can be convergent for certain cases. As an interesting example, we consider that $\mu_n = n-\frac{1}{2}$, $n=1,2,\dots$ and $T=\infty$, then finding the interpolating function is equivalent to solve the moment problem
\begin{equation}
    \int_0^{\infty} e^{-(n - \frac{1}{2}) t} \nu(t) dt = m_n.
\end{equation}
With a change of variable $x = e^{-t}$ and denote $\tilde{\nu}(x) = \nu(t)$, this is similar to the Hausdorff moment problem 
\begin{equation}\label{EQ: HAUSDORFF}
    \int_0^1 x^{n-1}\cdot  x^{-1/2}\tilde{\nu}(x) dx = m_n .
\end{equation}
\begin{definition}
For a given sequence $m_1, m_2, \dots$, we define
\begin{equation}
\begin{aligned}
     \lambda_{k, k'} &:= \binom{k}{k'}\sum_{l=0}^{k-k'} (-1)^{k-k'+l}  \binom{k-k'}{l} m_{k-l+1},\quad k\ge k' \ge 0.
\end{aligned}
\end{equation}
\end{definition}
\begin{lemma}[Widder~\cite{widder2015laplace}]
Let $\overline{L} > 0$ be an arbitrary fixed number, then the following condition
\begin{equation}\label{EQ: COND}
    (k+1) \sum_{k'=0}^{k} |\lambda_{k, k'}|^2 < \overline{L},\quad\forall k=0,1,\dots
\end{equation}
is necessary and sufficient for  ${x}^{-1/2}\tilde{\nu}(x)\in L^2[0,1]$ satisfying the moment problem~\eqref{EQ: HAUSDORFF}.
\end{lemma}
\begin{corollary}
If and only if the condition~\eqref{EQ: COND} is satisfied, the interpolation function $\nu(t)$ exists in $L^2[0, \infty)$ and 
\begin{equation}
    \int_0^{\infty} |\nu(t)|^2 dt  = \int_{0}^1 x^{-1}|\tilde{\nu}(x)|^2 dx < \infty
\end{equation}
\end{corollary}
Especially, if $m_{n} = {n}^{-1}$, one can compute directly that $(k+1) \sum_{n=0}^{k} |\lambda_{k, n}|^2 \equiv 1$ which permits a solution in $L^2[0,\infty)$. However, if we set $c_n = e^{-pn^{\alpha}}$ (same as Theorem~\ref{THM: MAIN 2}) and use the same initial condition that $$\sum_{n=1}^{\infty} |f_{n}|^2 < \infty,\;\text{ and }\;  m_{n} = e^{-(n-\frac{1}{2})\tau} \frac{f_{n}}{c_n} = e^{-(n-\frac{1}{2})\tau} e^{pn^{\alpha}}f_n, \quad n\ge 1.$$
Then for a given $k$, let $A$ denotes the upper triangle matrix 
\begin{equation}
    A_{l, l'} = \sqrt{k+1} \binom{k}{l}  e^{-(l'+\frac{1}{2})\tau} e^{(l'+1)^{\alpha} p} (-1)^{l'-l}  \binom{k-l}{l'-l},\quad 0\le l \le l'\le k
\end{equation}
and denote $\bff = (f_1, \dots, f_k)^T\in \ell^2$, then
\begin{equation}
\begin{aligned}
  \sup_{\|\bff\|=1} \sum_{k'=0}^k (k+1) |\lambda_{k, k'}|^2 &= \sup_{\|\bff\|=1}\bff^T A^TA \bff = \sup_{0\le l\le k} \sigma_{l}^2(A) \\&\ge \sup_{0\le l\le k} |A_{l,l}|^2 = \sup_{0\le l\le k}\left[\sqrt{k+1}\binom{k}{l} e^{-(l+\frac{1}{2})\tau}e^{(l+1)^{\alpha}p} \right]^2\\
  &\ge (k+1) e^{-\tau+2p}\to\infty,\quad \text{as}\;\; k\to\infty.
\end{aligned}
\end{equation}
This implies that there exists $\bff=(f_1, f_2, \dots, )^T\in \ell^2$ such that the moment problem has no solution in $L^2[0, \infty)$ as long as $\tau$ is finite. 
\section{Multiplicities}
In this section, we study the case that the eigenvalues have finite multiplicities. Denote the maximal multiplicity as $D$, we show that at most $D$ sampled trajectories $\{u_j\}_{j=1}^{D}$ are sufficient to span the solution subspace, that is, if $\tau > t_0$ then there exist interpolation functions $\nu_j\in L^2[t_0, t_1]$, $j=1,2,\dots, D$ that
\begin{equation}\label{EQ: MUL}
    w(x, \tau) = \sum_{j=1}^{D} \int_{t_0}^{t_1} u_j(x, t) \nu_j(t) dt
\end{equation}
Let the eigenvalues of $\cL$ be $0<\mu_1 <\mu_2 <\cdots$ without counting the multiplicities and the eigenfunctions $\{\phi_{n, l}\}$, $l=1,2,\dots, d_n$, are corresponding to the eigenvalue $\mu_n$, where $d_n\le D$ is the multiplicity of eigenvalue $\mu_n$, then for each sample solution $u_j$, we may write
\begin{equation}
    u_j(x, t)  = \sum_{n=1}^{\infty} e^{-\mu_n t} \sum_{l=1}^{d_n} b_{j,l}^n \phi_{n, l}(x), \quad j=1, 2, \ldots, D
\end{equation}
for some coefficients $\{ b_{j,l}^n \}$. Consider an arbitrary initial condition $f(x)\in L^2(\Omega)$ that
\begin{equation}
    f(x) = \sum_{n=1}^{\infty} \sum_{l=1}^{d_n} f_{n, l} \phi_{n, l}(x),
\end{equation}
the solution to~\eqref{EQ: EVOL} at time $\tau > t_0$ will be 
\begin{equation}
    w(x, \tau) = \sum_{n=1}^{\infty} e^{-\mu_n \tau} \sum_{l=1}^{d_n} f_{n, l} \phi_{n, l}(x).
\end{equation}
If the matrix $B_{n} = (b_{j,l}^n)_{jl}\in\bbR^{D\times d_n}$ is of full rank $d_n$ for each $n\ge 1$, there exist $p_{j}^n$, $j=1,2,\dots, D$ such that $     f_{n, l} = \sum_{j=1}^D p^n_{j} b_{j,l}^n, l=1, 2, \ldots, d_n $. It means there is a decomposition $f(x) = \sum_{j=1}^D f_j(x)$, where each $f_j$ is
\begin{equation}
    f_j(x) = \sum_{n=1}^{\infty} p^n_{j} \sum_{l=1}^{d_n} b_{j,l}^n  \phi_{n, l}(x).
\end{equation}
As a summary of above analysis, we have the following theorem as an analogue of Theorem~\ref{THM: MAIN 2} for eigenvalues with multiplicities.
\begin{theorem}\label{THM: MAIN 3}
Suppose the function $g_j(x)$ satisfies that
\begin{equation}
    g_j(x) = \sum_{n=1}^{\infty} \sum_{l=1}^{d_n} b^{n}_{jl} \phi_{n,l}(x),\quad j=1,2,\dots, D.
\end{equation}
Denote $B_n^{\dagger}$ as the Moore-Penrose inverse of $B_n$ and assume  $\|B_n^{\dagger}\|\le L e^{pn^{\alpha}}$, $\alpha\in[0, \beta)$, $n=1,2,\dots$.
Let $f\in L^2(\Omega)$ and denote $u_j(x, t)$, $w(x,t)$ the solutions with respect to initial conditions $g_j(x)$ and $f(x)$ respectively, then there exists an interpolation function $\nu_j(t; \tau, [t_0, t_1])\in L^2[t_0, t_1]$ that 
\begin{equation}\label{EQ: INT MUL}
    w(x, \tau) = \sum_{j=1}^D  \int_{t_0}^{t_1} u_j(x, t) \nu_j(t;\tau) dt,\quad \tau > t_0.
\end{equation}
\end{theorem}
\begin{proof}
Denote $\bp^n = (p_1^n,\dots, p_D^n)$ and $\bff^n = (f_{n,1}, \dots, f_{n,d_n})$, then $\bp^n = B_n^{\dagger}\bff^n$. Define 
$
\hat{\phi}_{j,n}(x)=\sum_{l=1}^{d_n} b^{n}_{jl} \phi_{n,l}(x), n=1,2,\ldots
$
which form an orthogonal set for each $j=1, 2,\ldots, D$. We have
\[
u_j(x,t)=\sum_{n=1}^{\infty}e^{-\mu_nt}\hat{\phi}_{j,n}(x), \quad w_j(x,t)=\sum_{n=1}^{\infty}e^{-\mu_nt}p_j^n\hat{\phi}_{j,n}(x), \quad w(x,t)=\sum_{j=1}^{D}w_j(x,t).
\]
Since $\|B_n^{\dagger}\|\le Le^{pn^{\alpha}}$, 
\begin{equation}
    |p_j^n|\le \|\bp^n\| \le Le^{{p}n^{\alpha}} \|\bff^n\|\le  Le^{{p}n^{\alpha}} \|f\|_{L^2(\Omega)} \le C e^{{p}n^{\alpha}},\quad j=1,2,\dots, D,
\end{equation}
according to the Theorem~\ref{THM: MAIN 2}, there exist $\nu_j\in L^2[t_0, t_1]$, $w_j(x, \tau) =  \int_{t_0}^{t_1} u_j(x, t) \nu_j(t;\tau) dt$.
\end{proof}
\begin{remark}
The above theorem requires finite multiplicities of eigenvalues to obtain the exact interpolation~\eqref{EQ: INT MUL} which may not be true in general, for instance, $-\Delta$ in 2D unit square. However instead of producing an accurate representation, we usually only need to find an approximation. If the solution trajectories $\{u_i\}_{i=1}^D$ can capture the subspace spanned by the leading eigenfunctions up to a small tolerance, the approximation suffices for practical uses. 
\end{remark}

\section{Application}
\subsection{Data driven model reduction}
In the case that snapshots of a sample solution trajectory can be superposed to approximate any snapshot of an arbitrary solution, one can use a data driven approach for model reduction and approximation of new solution bypass explicit learning of the underlying PDE. For example, if the sample solution $u(x,t)$ is observed on the time interval $[t_0,t_1]$, then in theory one can approximate any solution at a later time $\tau > t_0$ using certain superposition of the observed sample solution snapshots. In particular, if one can find a finite dimensional space $V$ to which the sample solution trajectory is close, 
\begin{equation}
\label{EQ: approximation}
    \|u(x,t) - P_V u(x,t)\|_{L^2(\Omega)}\le \eps \|u(x, 0)\|_{L^2(\Omega)}, \quad t\in [t_0, t_1], 
\end{equation}
where $P_V$ is the projection operator onto $V$, then any solution $w(x,\tau),\tau>t_0$ is close to $V$ as well since 
\begin{equation}\label{EQ: W ERR}
\begin{aligned}
     \left\|w(x, \tau) - P_V w(x,\tau) \right\|_{L^2(\Omega)} &= \left\|\int_{t_0}^{t_1} u(x, t)\nu(t) dt - \int_{t_0}^{t_1} P_V u(x, t)\nu(t;\tau) dt\right\|_{L^2(\Omega)}\\
     \\ &\le \int_{t_0}^{t_1} \left\|(u(x, t) - P_Vu(x,t) )\nu(t;\tau)\right\|_{L^2(\Omega)} dt \\
     &\le\eps \|u(x, 0)\|_{L^2(\Omega)} \int_{t_0}^{t_1} |\nu(t;\tau)|dt\\
     &\le \eps \sqrt{(t_1 - t_0)} \|\nu(\cdot\,;\tau)\|_{L^2[t_0, t_1]}\|u(x, 0)\|_{L^2(\Omega)}.
\end{aligned}
\end{equation}
However, fixing $t_0$ and $t_1$ and let $\tau\to t_0^{+}$, the norm $\|\nu(\cdot\,; \tau)\|_{L^2[t_0, t_1]}$ will increase rapidly according to Theorem~\ref{THM: MAIN 2}. For $u_t=-\cL u$, where $\cL$ is a self-adjoint elliptic operator, it has been shown~\cite{he2022much} that given any $\eps > 0$, for any solution trajectory $u(x,t)$ on $[t_0, t_1]$,  there exists a linear subspace $V\subset L^2(\Omega)$ of dimension $\dim(V) = \cO(|\log\eps|\log(t_1/t_0))$ such that \eqref{EQ: approximation} is satisfied. 

To find a discrete approximation of the subspace $V$ for a sample solution $u(x,t)$, one can observe $u(x_i,\tau_j)$ on a space time grid, $(x_i,\tau_j)$, with grid size $h,\Delta t$ in space and time respectively such that, for any $j$ and $\forall t \in [\tau_j, \tau_{j+1})$, $\|u(\cdot,t)-I_t[u(\cdot,\tau_j),u(\cdot,\tau_{j+1})]\|_{L^2(\Omega)}=O(\Delta t^2)=O(\eps) $ and $\|u(\cdot,\tau_j)-I_x[u(x_i,\tau_j)]\|_{L^2(\Omega)}=O(h^2)=O(\eps)$, where $I_t$ and $I_x$ are linear interpolation operators in time and space respectively.
%Numerically, if the sample solution is observed at time instants $t_0=\tau_1<\tau_2<\dots < \tau_{M} = t_1$, $\Delta t = \frac{t_1 - t_0}{M-1}$ and at grid points in space,  
Let $U_{ij} = u(x_i, \tau_j)$ denote the solution matrix whose singular value decomposition (SVD) is $U = P \Sigma Q^{\ast}$. $V$ can be approximated by the finite dimensional linear space spanned by the leading left-singular vectors from columns of $P$. Once $V$ is found, it can be used for model reduction and other applications. For example, an arbitrary solution can be approximated well by a linear combination of the basis of $V$. The coefficients can be determined by a few measurements, e.g., at a few locations or some integral quantities,  of the new solution without knowing or solving the underlying PDE (see an example in~\ref{SEC: EXP 5}).

\subsection{Noisy data}\label{SEC: NOISE}
In practice, the observed solution data $u(x_i,\tau_j)$ could have some measurement errors. Hence it is natural to ask if the subspace $V$ is stable under such perturbations. For that reason, we assume the noisy sampled solution matrix $$\tilde{U}_{ij} = \tilde{u}(x_i,\tau_j) = u(x_i,\tau_j) + e(x_i, \tau_j),$$ where $u(x_i,\tau_j)$ denotes the exact solution data and the random perturbations $e(x_i, \tau_j)$ are i.i.d mean-zero random variables with variation $\delta^2$. We approximate $\tilde{V}$ from the corresponding singular vectors of $\tilde{U}$ such that $\dim(\tilde{V}) = \dim(V) = \upsilon$. Suppose $U = P \Sigma Q^{\ast}$ and $\tilde{U} = \tilde{P}\tilde{\Sigma} \tilde{Q}^{\ast}$, the perturbation matrix $E = \tilde{U} - U$, then from the Wedin's theorem~\cite{wedin1972perturbation}, 
\begin{equation}\label{EQ: WEDIN}
    \| \sin \Theta(V, \tilde{V}) \|_{F}\le \frac{\sqrt{2\upsilon}}{\ell} \|E\|_{F},
\end{equation}
where $\Theta(V, \tilde{V})$ denotes the canonical angles between $V$ and $\tilde{V}$, the left-hand side measures the difference between the projection mappings onto $V$ and $\tilde{V}$, respectively, see~\cite{stewart1998perturbation} and references therein for detailed discussions. Let $\sigma_i$ and $\tilde{\sigma}_i$ be the $i$-th singular value of $U$ and $\tilde{U}$, respectively, the parameter $$\ell := \min\{ \min_{1\le i\le\upsilon,1\le j\le N -\upsilon} |\sigma_{i} - \tilde{\sigma}_{\upsilon + j}|, \min_{1\le i\le \upsilon} \sigma_i \}. $$ 
Particularly, the Bernoulli random perturbation has been studied in~\cite{vu2011singular} and Gaussian random perturbation is considered in~\cite{wang2015singular}. For $u_t=-\cL u$, where $-\cL$ is a self-adjoint elliptic operator, given any $\eps>0$ and any solution trajectory $u(x,t)$ on $[t_0, t_1]$,  there exists a linear subspace $V\subset L^2(\Omega)$ of dimension $\dim(V) = \cO(|\log\eps|\log(t_1/t_0))$ such that \eqref{EQ: approximation} is satisfied~\cite{he2022much}. This means $\sigma_n = O(e^{-n})$ which makes $\ell$ decay very fast with respect to $\upsilon$ and the computation of $V$ is sensitive to noise.

When the resolution in time is sufficiently fine, due to the linearity, instead of taking the singular value decomposition directly on the sample solution data $\tilde{u}(x_i, \tau_j)$, we may first regularize the data by averaging on a time window $[\tau_j, \tau_j+S\Delta t]$ 
\begin{equation}\label{EQ: V}
\begin{aligned}
    {v}(x_i, \tau_j) &= \frac{1}{S+1}\sum_{s=0}^S\tilde{u}(x_i, \tau_{j+s}) = \frac{1}{S+1}\sum_{s=0}^S {u}(x_i, \tau_{j+s}) + \tilde{e}_{i,j},
\end{aligned}
\end{equation}
where the variation of $\tilde{e}_{i,j}$ is $\text{Var}\left[\frac{1}{S+1}\sum_{s=0}^S e(x_i, \tau_{j+s})\right] = \frac{\delta^2}{S+1}$. This process is equivalent to sampling the averaged data 
\begin{equation}\label{EQ: MODIFIED}
    v(x, t) = \sum_{n=1}^{\infty} e^{-\mu_n t}  \left(\frac{1}{(S+1)}\frac{1 - e^{-(S + 1)\mu_n \Delta t} }{1 - e^{-\mu_n \Delta t}} \right) c_n \phi_n(x) + \tilde{e}(x, t),
\end{equation}
where $\tilde{e}(x,t)$ is a random variable for each $(x,t)$ with variance $\frac{\delta^2}{S+1}$. Therefore as $\Delta t\to 0$ is sufficiently fine, one can fix $r>0$ and take $S = r\Delta t^{-1}\to\infty$ to reduce the noise variations, then we perform the singular value decomposition on the modified solution~\eqref{EQ: MODIFIED} which has a smaller error bound in~\eqref{EQ: WEDIN}. The above modified sampled solution trajectory becomes
\begin{equation}
    \lim_{S\to\infty} v(x, t)\stackrel{P}{=} \sum_{n=1}^{\infty} e^{-\mu_n t}  \left(\frac{1 - e^{-r\mu_n} }{r\mu_n} \right) c_n \phi_n(x),
\end{equation}
which averages the ${u}(x,t)$ in a window $(t, t+r)$, the coefficient $\frac{1 - e^{-r\mu_n} }{r\mu_n}c_n$ still satisfies the condition in Theorem~\ref{THM: MAIN 2}, although the factor $\frac{1}{\mu_n}$ will make the interpolation function $\|\tilde{\nu}(t)\|_{L^2[t_0,t_1]}$ larger than using the true $u(x,t)$ data. As a summary, if the solution data are finely sampled and contain mean-zero noises, we can simply perform a local average on finely sampled solution and extract the subspace from the smoothed solution.

\section{Numerical experiments}
The following experiments are computed with MATLAB 2016a. The experiment source code is hosted at GitHub~\footnote{\href{https://github.com/lowrank/pde-subspace}{https://github.com/lowrank/pde-subspace}}.
\subsection{Experiment 1}\label{EX: 1}
In this experiment, we take the elliptic operator  $\cL u = -u_{xx}$ on $[0, 1]$ with Dirichlet boundary condition. The eigenvalue $\mu_n = \pi^2 n^2$, $n=1,2,\dots$, which satisfies the Assumption~\ref{AS: 1} with $\beta=2$ and Assumption~\ref{AS: 2}. We take the sample solution as 
\begin{equation}
    u(x, t) = \sum_{n=1}^{\infty} \frac{(-1)^n}{n^2} e^{-n^2\pi^2 t} \sin n \pi x
\end{equation}
on the time span $t\in[10^{-6},  1]$ and the subspace $V$ is computed by selecting the leading left-singular vectors above the singular value threshold $10^{-12}$ and $\dim(V) = 27$ in this example.
It should be pointed out that the singular vectors are not necessarily close to eigenfunctions, see Fig~\ref{FIG: SING VEC}.
\begin{figure}[!htb]
    \centering
    \includegraphics[scale=0.5]{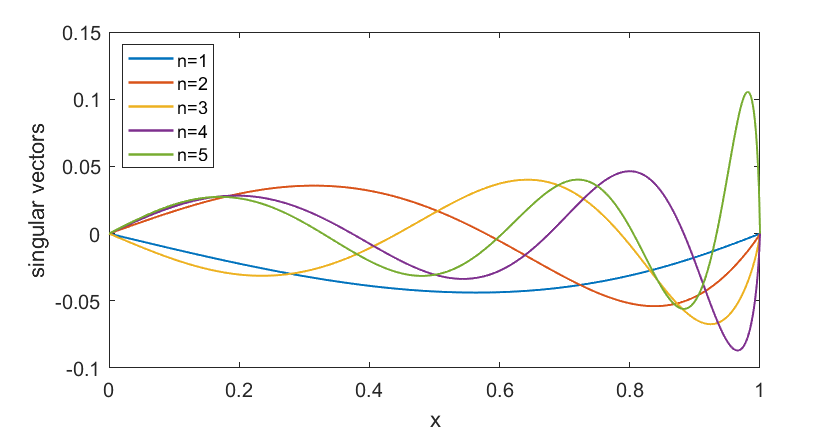}
    \caption{The first five singular vectors of sampled solution data.}
    \label{FIG: SING VEC}
\end{figure}

We then validate Theorem~\ref{THM: MAIN 2} with solutions corresponding to eigenmodes
\begin{equation}
    w_n(x, \tau) = e^{-n^2 \pi^2 \tau} \sin n\pi x,\quad 1\le n \le 8
\end{equation}
at time $\tau  = 0.1$. If $\tau$ is very close to $t_0=10^{-6}$, the norm of  $\|\nu\|_2$ will be very large by Theorem~\ref{THM: MAIN 2} which leads to a relatively large constant in the approximation error~\eqref{EQ: W ERR}. For each $n$, we evaluate the relative error by 
\begin{equation}\label{EQ: ETA_N}
    \eta = \frac{\|w_n(x, \tau) - P_V w_n(x,\tau)\|}{\|w_n(x,\tau)\|}.
\end{equation}
We summarize the experiment result  in the following Tab~\ref{TAB: ERR}, where the eigenmodes $w_n(x, \tau)$ can be approximated with quite small relative errors by the subspace $V$.
\begin{table}[!htb]
    \centering
    \caption{Relative error of approximation by subspace $V$.}
    \vspace{0.2cm}
    \begin{tabular}{|c || c | c | c | c |}
    \hline
        $n$ &  1 & 2 & 3& 4  \\
        \hline
        $\eta$ & $2.10\times 10^{-16}$ & $5.09\times 10^{-16}$ & $6.91\times 10^{-16}$ & $1.35\times 10^{-16}$  \\
        \hline
        $n$ & 5 & 6 & 7 & 8\\
        \hline
        $\eta$ & $4.75\times 10^{-16}$ & $3.76\times 10^{-16}$& $2.81\times 10^{-15}$ & $3.66\times 10^{-13}$\\
        \hline
    \end{tabular}
    \label{TAB: ERR}
\end{table}

\subsection{Experiment 2}
In this experiment, we take the operator $\cL u = -u_{xx}+ u$ with periodic boundary condition on $[0, 1]$, then the multiplicity of each eigenvalue $\mu_n = n^2+1$ is two except for $n=0$ which corresponds to the constant eigenfunction. In this case, according to the Theorem~\ref{THM: MAIN 3}, we will need two sample solution trajectories with linearly independent coefficients on the time span $t\in [10^{-6} , 1]$:
\begin{equation}
\begin{aligned}
    u_1(x, t) &= \sum_{n=1}^{\infty} \frac{(-1)^n}{n^2} e^{-(n^2\pi^2+1)t} \sin n\pi x,  \\
     u_2(x, t) &= e^{-t} + \sum_{n=1}^{\infty} \frac{(-1)^n}{n^2} e^{-(n^2\pi^2+1)t} \cos n\pi x. 
\end{aligned}
\end{equation}
 Similar to the previous experiment, the solution subspace $V$ is union of the singular vectors from singular value decomposition of $u_1$ and $u_2$ truncated at the threshold of $10^{-12}$ and $\dim(V) = 52$. We validate Theorem~\ref{THM: MAIN 3} with eigenmodes
\begin{equation}
    w_n(x, \tau) = e^{-(n^2 \pi^2 + 1)\tau} e^{in\pi x},\quad 0\le n \le 8
\end{equation}
at time $\tau =0.1$. The relative errors are computed by~\eqref{EQ: ETA_N} and summarized in Tab~\ref{TAB: ERR 2}.
\begin{table}[!htb]
    \centering
    \caption{Relative error of approximation by subspace $V$.}
    \vspace{0.2cm}
    \begin{tabular}{|c || c | c | c | c |}
    \hline
        $n$ &  1 & 2 & 3& 4  \\
        \hline
        $\eta$ & $5.82\times 10^{-16}$ & $8.55\times 10^{-16}$ & $9.28\times 10^{-16}$ & $2.98\times 10^{-16}$  \\
        \hline
        $n$ & 5 & 6 & 7 & 8\\
        \hline
        $\eta$ & $6.56\times 10^{-16}$ & $4.64\times 10^{-16}$& $4.19\times 10^{-15}$ & $5.54\times 10^{-13}$\\
        \hline
    \end{tabular}
    \label{TAB: ERR 2}
\end{table}

\subsection{Experiment 3}\label{SEC: EX 3}
In contrast to one dimensional case, we demonstrate in 2D that the subspace spanned by a single solution trajectory of second order parabolic equation's cannot be used to approximate all solutions well. For simplicity, we take the elliptic operator
\begin{equation}
    \cL u = -\Delta u 
\end{equation}
on a rectangular domain $[0, 1]\times [0, 2^{-1/4}]$ with Dirichlet boundary condition. The eigenvalues are $\lambda_{m, n} = \pi^2(m^2 + \sqrt{2}n^2)\in \pi^2 \bbZ[\sqrt{2}]$, $m\ge 1, n\ge 1$. It can be shown the eigenvalues are simple and grow with rate $\beta = n/d = 1$ which violates the Assumption~\ref{AS: 1}. On the other hand, if 
$m_1^2 + \sqrt{2} n_1^2 <  m_2^2 + \sqrt{2} n_2^2$,
then the difference
\begin{equation}\label{EQ: EXP 3 INEQ}
    \left|     m_1^2 + \sqrt{2} n_1^2 -  (m_2^2 + \sqrt{2} n_2^2) \right| \ge \frac{c}{\sqrt{2}({  m_2^2 + \sqrt{2} n_2^2 })}
\end{equation}
for certain $c > 0$. The proof is found in Lemma~\ref{LEM: APP 1}. That is to say the eigenvalue gap $\mu_{n+1} - \mu_n \ge \frac{c\pi^2}{\sqrt{2} \mu_{n+1}} = \Theta(n^{-1})$, hence satisfies the Assumption~\ref{AS: 2}. 

We take the sample solution on the time span $t\in[10^{-6}, 1]$ in the following expansion form,
\begin{equation}
\begin{aligned}
    u(x, y, t) &= \sum_{m=1}^{\infty} \sum_{n=1}^{\infty}\frac{1}{m^2n^2} e^{-\pi^2(m^2 + \sqrt{2}n^2) t} \sin(m\pi x) \sin(2^{1/4} n\pi y) \\
    &= \sum_{m=1}^{\infty} \frac{1}{m^2} e^{-\pi^2 m^2 t} \sin(m\pi x) \sum_{n=1}^{\infty} \frac{1}{n^2} e^{-\sqrt{2}\pi^2 n^2 t} \sin(2^{1/4} n\pi y).
\end{aligned}
\end{equation}
Similarly, we formulate the data matrix by evaluating the solution uniformly in both space and time, the subspace $V$ consists of the leading left-singular vectors above the singular value threshold $10^{-12}$ and $\dim(V) = 34$. The subspace is validated against the first 8 eigenmodes
\begin{equation}
    w_{m,n} (x, y, \tau)= e^{-\pi^2(m^2 + \sqrt{2}n^2)\tau} \sin(m\pi x)  \sin( 2^{1/4} n \pi y)
\end{equation}
at time $\tau  = 0.1$. The relative errors are summarized in the Tab~\ref{TAB: ERR 3}. It can be seen that the subspace spanned by the sample solution trajectory cannot even approximate the first few eigenmodes well.
\begin{table}[!htb]
    \centering
    \caption{Relative error of approximation by subspace $V$.}
    \vspace{0.2cm}
    \begin{tabular}{|c || c | c | c | c |}
    \hline
        $(m, n)$ &  $(1,1)$ & $(2,1)$ & $(1,2)$& $(2,2)$  \\
        \hline
        $\eta$ & $4.02\times 10^{-14}$ & $5.20\times 10^{-14}$ & $9.14\times 10^{-15}$ & $2.03\times 10^{-12}$  \\
        \hline
        $(m, n)$ & $(3,1)$ & $(1,3)$ & $(3, 2)$ & $(2, 3)$\\
        \hline
        $\eta$ & $5.47\times 10^{-12}$ & $1.84\times 10^{-8}$& $2.93\times 10^{-6}$ & $1.05\times 10^{-4}$\\
        \hline
    \end{tabular}
    \label{TAB: ERR 3}
\end{table}
% \setlength\extrarowheight{5pt}
% \begin{table}[!htb]
%     \centering
%         \caption{Relative error of approximation to eigenmodes by subspace $V$.}
%         \vspace{0.2cm}
%     \begin{tabular}{|c|c|c|c|c|c|}
%     \hline
%         \diagbox[height=1.5\line]{$m$}{$n$} & 1 & 2 & 3 &4 & 5\\
%         \hline
%         1 & $6.65\times 10^{-14}$ & $1.28\times 10^{-9}$ & $4.40\times 10^{-3}$ & $7.31 \times 10^{-2}$  & $1.96\times 10^{-1}$\\
%                 \hline 
%         2 & $2.70\times 10^{-11}$ & $6.07\times 10^{-5}$& $4.88\times 10^{-1}$& $6.59\times 10^{-1}$& $6.92\times 10^{-1}$ \\
%                 \hline 
%         3 & $7.37\times 10^{-5}$ & $2.08\times 10^{-1}$ & $6.87\times 10^{-1}$& $7.37\times 10^{-1}$& $7.89\times 10^{-1}$ \\
%                 \hline 
%         4 & $6.46\times 10^{-2}$ & $5.96\times 10^{-1}$& $7.72\times 10^{-1}$& $8.97\times 10^{-1}$ & $9.31\times 10^{-1}$ \\
%                 \hline 
%         5 & $1.69\times 10^{-1}$ & $6.15\times 10^{-1}$& $9.02\times 10^{-1}$& $9.02\times 10^{-1}$& $9.25\times 10^{-1}$ \\
%         \hline
%     \end{tabular}
%     \label{TAB: ERR 3}
% \end{table}

\subsection{Experiment 4}
In two dimension, we consider the 4th order elliptic operator $\cL = \partial_{x}^{(4)} + \partial_y^{(4)}$ on the rectangular domain $\Omega = [0, 1]\times [0, 2^{-1/8}]$ with boundary conditions: $u = 0$ on $\partial\Omega$, $u_{xx}=0$ on $\{x=0, x=1\}$ and $u_{yy} = 0$ on $\{y=0, y=2^{-1/8}\}$. The eigenvalues are $\lambda_{m,n} = \pi^2 (m^4 + \sqrt{2}n^4)$, $m, n\in\bbZ_{+}$. The eigenvalues satisfies both assumptions~\ref{AS: 1} and~\ref{AS: 2} by deriving an analogue of the inequality~\eqref{EQ: EXP 3 INEQ}.

The sampled solution on the time span $t\in[10^{-6},1]$ is the following form
\begin{equation}
\begin{aligned}
    u(x, y, t) 
    = \sum_{m=1}^{\infty} \frac{1}{m^2} e^{-\pi^2 m^4 t} \sin(m\pi x) \sum_{n=1}^{\infty} \frac{1}{n^2} e^{-\sqrt{2}\pi^2 n^4 t} \sin(2^{1/8} n\pi y).
\end{aligned}
\end{equation}
The subspace $V$ consists of the left-singular vectors above the singular value threshold $10^{-12}$ and $\dim(V) = 23$. We validate the subspace against the first 8 eigenmodes
\begin{equation}\label{EQ: 4TH}
    w_{m,n}(x, y, \tau) = e^{-\pi^2 (m^4 + \sqrt{2}n^4)\tau} \sin(m\pi x) \sin(2^{1/8}n\pi y)
\end{equation}
at $\tau = 0.1$. The relative errors are summarized in the Tab~\ref{TAB: ERR 4}. Unlike the second order case, the eigenmodes~\eqref{EQ: 4TH} are well resolved by $V$.
\begin{table}[!htb]
    \centering
    \caption{Relative error of approximation to eigenmodes by subspace $V$.}
    \vspace{0.2cm}
    \begin{tabular}{|c || c | c | c | c |}
    \hline
        $(m, n)$ &  $(1,1)$ & $(2,1)$ & $(1,2)$& $(2,2)$  \\
        \hline
        $\eta$ & $3.46\times 10^{-15}$ & $3.37\times 10^{-14}$ & $3.31\times 10^{-14}$ & $2.35\times 10^{-14}$  \\
        \hline
        $(m, n)$ & $(3,1)$ & $(3,2)$ & $(1, 3)$ & $(2, 3)$\\
        \hline
        $\eta$ & $3.61\times 10^{-14}$ & $1.62\times 10^{-12}$& $1.26\times 10^{-12}$ & $3.66\times 10^{-11}$\\
        \hline
    \end{tabular}
    \label{TAB: ERR 4}
\end{table}
% \begin{assumption}
% Suppose the normalized eigenfunctions satisfy
% \begin{equation}
%       \|\phi_n\|_{L^1(\Omega)} \ge Q_1 n^{-\delta},\quad \|\phi_n\|_{L^{\infty}(\Omega)}  \le Q_2 n^{\gamma},\quad n=1,2,\dots.
% \end{equation}
% \end{assumption}
% Under the above assumption, we show that with high probability the coefficients $\{ c_n\}$ are satisfying the condition that 
% \begin{equation}
%   |c_n| =  |{\phi_n(x')} | \ge C n^{-p}
% \end{equation}
% where $p > \delta$.
% This is because
% \begin{equation}
%     \int_{0}^{Q_2 n^{\gamma}} \text{meas}(\{|\phi_n(x')| > \alpha\}) d\alpha = \|\phi_{n}(x')\|_{L^1(\Omega)} \ge Q_1 n^{-\delta}
% \end{equation}
% which implies
% \begin{equation}
%     \text{meas}(\{|\phi_n(x')| > C n^{-p}\})   > \frac{Q_1 n^{-\delta}-   C n^{-p} |\Omega|}{Q_2 n^\gamma}
% \end{equation}

\subsection{Experiment 5}\label{SEC: EXP 5}
One useful application for data driven model reduction is to predict a PDE solution from limited observed data, e.g., by a few local sensors. Suppose that a small linear space $V$ has been found from a sample solution. One can then approximate a new solution by a linear combination of basis of $V$, where the coefficients can be determined, e.g., least square fitting, by a few measurements of the new solution.   Let us take the subspace $V$ from the experiment in~\ref{EX: 1}. If we observed certain solution at time $\tau$ 
\begin{equation}
    \tilde{u}(x, \tau) = \sum_{n=1}^{\infty} \omega_{n} e^{-\pi^2 n^2 \tau} \sin(\pi n x)
\end{equation}
at locations $z_i\in [0, 1]$, $i=1,2,\dots, 50$ (uniformly distributed), we can use the subspace $V$, whose orthonormal basis is available, to find out the solution by a direct least square fitting at $\{u(z_i, \tau))\}_{i=1}^{50}$. In this experiment, we set 
\begin{equation}
    \omega_n = \begin{cases}
    \text{i.i.d $\cU(-1, 1)$ random variables}, &n=1,2,\dots, 10^3, \\
    0,&\text{otherwise}.
    \end{cases}
\end{equation}
For each value of $\tau$, we sample $10^3$ realizations of $\tilde{u}$ and compute the approximated solutions from $V$.  The recorded average relative errors are shown in Fig~\ref{FIG: INTERP}. When $\tau$ decreases, the constant in the estimate~\eqref{EQ: W ERR} grows very fast, thus we can see the error is quite large for small value of $\tau$.
\begin{figure}[!htb]
    \centering
    \includegraphics[scale=0.75]{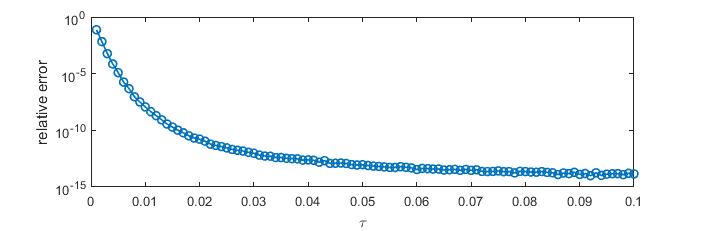}
    \caption{Average relative $L^2$ error for noiseless data with respect to different $\tau$.}
    \label{FIG: INTERP}
\end{figure}
\subsection{Experiment 6}
We use the same setting of Experiment 5 but with noises. Suppose the equally spaced discretization of time $\{t_j\}_{j\ge 1}$ with time step $\Delta t$ varying from $10^{-8}$ to $10^{-3}$. Spatially the discretization $\{x_i\}_{i\ge 1}$ is equally spaced that $\Delta x = 10^{-3}$.  The discretized sample solution data is
\begin{equation}
        \tilde{u}(x_i, t_j) = \sum_{n=1}^{\infty} \frac{(-1)^n}{n^2} e^{-n^2\pi^2 t_j} \sin n \pi x_i + e(x_i, t_j),
\end{equation}
where $e(x_i, t_j)$ is i.i.d uniformly sampled in $[-10^{-3}, 10^{-3}]$.  The averaging is taken on a window size of $S = \lceil 0.1/\Delta t \rceil$, see Sec~\ref{SEC: NOISE}. We perform the same experiment as in Experiment 5 on $\tilde{u}$. The resulting relative errors are recorded in Fig~\ref{FIG: INTERP NOISE}. The random noises in the sampled solution introduce quite large errors in the reconstructed solution compared to the noiseless Experiment 5 since the parameter $\ell\approx 10^{-12}$ in~\eqref{EQ: WEDIN} amplifies the error greatly.

\begin{figure}[!htb]
    \centering
    \includegraphics[scale=0.7]{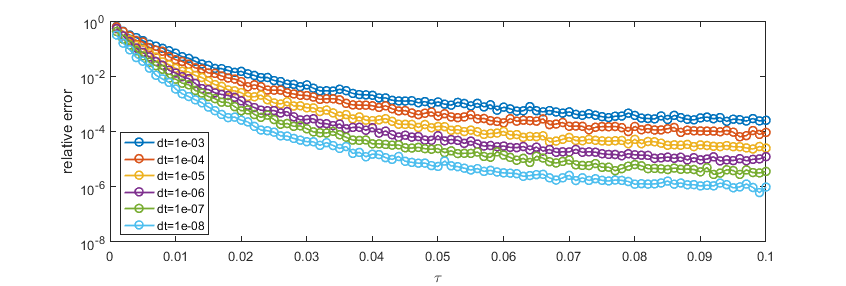}
    \label{FIG: INTERP NOISE}
    \caption{Average relative $L^2$ error for noise data with respect to different $\tau$ and $\Delta t$.}
\end{figure}

\section{Conclusion}
In this work, time evolution equation is used as an example to show if the space spanned by all snapshots of a single (or multiple) solution(s) is as rich as the space spanned by the snapshots of all solutions. It is shown that when smoothing of the differential operator dominates the diversity in spatial dimensions, i.e., the order of the PDE is larger than the space dimension, one can use superposition of snapshots of a single (or multiple) solution(s) to construct an arbitrary solution. In practice, this means that if one can find a linear space a single (or multiple) sampled solution(s) is (are) close to, that space can be used to approximate all solutions. 
%self-adjoint parabolic equation~\eqref{EQ: EVOL}'s solution can be linearly interpolated by $D$ measured solution trajectories, where $D$ is the maximal multiplicity of eigenvalues if the order $N_{\cL}$ is greater than the dimension $d$. However it remains unclear how to precisely characterize the counterpart that the order is less than or equal to the dimension. An immediate application is to interpolate an arbitrary solution from its local measurements once the solution subspace has been approximately obtained. On the other hand, 
Moreover, the learned solution space can be used as a general regularization for PDE learning~\cite{he2022much} which complements the limitations of local matching. 
%Moreover, when the solution subspace is less sensitive than the solution itself upon the equation parameters, it is possible to combine the technique~\cite{chen2009subspace} to PDE identification for parabolic type. 

\section*{Acknowledgment}
H. Zhao is partially supported by NSF grant DMS-2012860.

\appendix
\section{Auxiliary lemmas}

\begin{lemma}\label{LEM: APP 1}
Suppose $m_1, n_1, m_2, n_2\in\bbZ_{+}\cup \{0\}$ and if $m_1^2 + \sqrt{2} n_1^2 < m_2^2 + \sqrt{2} n_2^2$, then 
\begin{equation}
    \left|(m_2^2 + \sqrt{2} n_2^2) - (m_1^2 + \sqrt{2} n_1^2)\right| >  \min\left( \frac{c}{n_2^2}, \frac{c}{\sqrt{2}(m_2^2)}\right).
\end{equation}
where $c\in(0, 1)$ is a constant. In particular, we have 
\begin{equation}
     \left|(m_2^2 + \sqrt{2} n_2^2) - (m_1^2 + \sqrt{2} n_1^2)\right| >  \frac{c}{\sqrt{2} ({m_2^2 + \sqrt{2}n_2^2})}.
\end{equation}
\end{lemma}
\begin{proof}
We consider 3 cases. (1). $m_2 \ge m_1$ and $n_2 \ge n_1$, since both equal signs cannot hold simultaneously, we obtain the trivial bound
\begin{equation}
     \left|(m_2^2 + \sqrt{2} n_2^2) - (m_1^2 + \sqrt{2} n_1^2)\right| \ge 1.
\end{equation}
(2). $m_1 \ge m_2$, $n_2 > n_1$. 
According to the Liouville's theorem, for any $p, q\in\bbZ_{+}$, there exists a constant $c > 0$ that
\begin{equation}
    |q\sqrt{2} - p| > \frac{c}{q}.
\end{equation}
therefore 
\begin{equation}
\begin{aligned}
\left|m_2^2 + \sqrt{2} n_2^2 - ( m_1^2 + \sqrt{2} n_1^2)\right| &= \left|(m_1^2 - m_2^2) - \sqrt{2} (n_2^2 - n_1^2)\right| \\ &> \frac{c}{|n_2^2 - n_1^2|} \ge \frac{c}{n_2^2}.
\end{aligned}
\end{equation}
(3). $m_1 < m_2$, $n_2 \le n_1$. Use the same argument as above,
\begin{equation}
\begin{aligned}
\left|m_2^2 + \sqrt{2} n_2^2 - ( m_1^2 + \sqrt{2} n_1^2)\right| &= \frac{1}{\sqrt{2}}\left| 2(n_1^2 - n_2^2)  -\sqrt{2}(m_2^2 - m_1^2)\right| \\ &> \frac{c}{\sqrt{2}|m_2^2 - m_1^2|} \ge \frac{c}{\sqrt{2}(m_2^2)}.
\end{aligned}
\end{equation}
By taking the minimum of the three cases, we obtain the desired result.
\end{proof}

\section{Proof of Lemma~\ref{LEM: EST}}\label{AP: LEM}
\begin{proof}
Let $N(x)$ be the counting function for the eigenvalues $\{\mu_k\}_{k\ge 1}$ that 
\begin{equation}
    N(x) = \begin{cases}
    0,\quad  0 \le x < \mu_1, \\
    k, \quad \lambda_k \le x < \lambda_{k+1}.
    \end{cases}
\end{equation}
Then from the estimate in assumption \textbf{A1}, we find 
\begin{equation}\label{EQ: ERR}
    0 \le  M^{-1/\beta} x^{1/\beta} + \delta(x) - N(x)\le 1,
\end{equation}
where $\delta(x) = o(x^{(1-\sigma)/\beta})$ as $x\to \infty$. Then use Riemann-Stieljes integral,
\begin{equation}
\begin{aligned}
        &\log \left( \prod_{j\ge 1} \left(1 + \frac{\mu_n}{\mu_j}\right) \right) = \sum_{j\ge 1} \log \left( 1 + \frac{\mu_n}{\mu_j}\right) \\&= \int_{\mu_1^{-}}^{\infty} \log \left( 1 + \frac{\mu_n}{x}\right) d N(x) = \mu_n \int_{\mu_1}^{\infty} \frac{N(x) dx}{x(x+\mu_n)}
\end{aligned}
\end{equation}
The error bound of~\eqref{EQ: ERR} implies there exists a constant $C > 0$ that
\begin{equation}
    \left|  \mu_n \int_{\mu_1}^{\infty} \frac{N(x) dx}{x(x+\mu_n)} - M^{-1/\beta} \mu_n \int_{\mu_1}^{\infty} \frac{x^{1/\beta}dx}{u(u+\mu_n)}\right|\le C \mu_n \int_{\mu_1}^{\infty} \frac{x^{(1-\sigma)/\beta} dx}{x(x+\mu_n)} .
\end{equation}
The integral on right-hand side is
\begin{equation}\label{EQ: C}
\begin{aligned}
        C \mu_n \int_{\mu_1}^{\infty} \frac{x^{(1-\sigma)/\beta} dx}{x(x+\mu_n)} &= C\mu_n^{(1-\sigma)/\beta} \int_{\mu_1/\mu_n}^{\infty} \frac{d y}{y^{1 - (1-\sigma)/\beta}(y+1)}\\
    &=\begin{cases}
        C\mu_n^{(1-\sigma)/\beta} \left( \int_{0}^{\infty} \frac{d y}{y^{1 - (1-\sigma)/\beta}(y+1)} +o(1) \right),&\quad 0 < \sigma < 1 \\
        C \log\frac{\mu_n+\mu_1}{\mu_1} &\quad \sigma= 1
    \end{cases} \\
    &= o(\mu_n^{1/\beta}).
\end{aligned}
\end{equation}
By a similar argument, we derive
\begin{equation}\label{EQ: M}
    M^{-1/\beta} \mu_n \int_{\mu_1}^{\infty} \frac{x^{1/\beta}dx}{x(x+\mu_n)} = M^{-1/\beta} \mu_n^{1/\beta} \left( \int_{0}^{\infty} \frac{d y}{y^{1 - 1/\beta}(y+1)} +o(1) \right).
\end{equation}
Then we absorb~\eqref{EQ: C} into~\eqref{EQ: M} and immediately obtain the first product in~\eqref{EQ: PROD} is
\begin{equation}
\begin{aligned}
    \prod_{j\ge 1} \left(1 + \frac{\mu_n}{\mu_j}\right) =\exp \left(  M^{-1/\beta} \mu_n^{1/\beta}  (\zeta_{0, \beta} + o(1)) \right),\quad n\to\infty.
\end{aligned}
\end{equation}
where $\zeta_{a, b}$ denotes the following integral, 
\begin{equation}
    \zeta_{a, b} := \int_0^{\infty} \frac{dy}{y^{1 - a/b} (y+1)},\quad 0\le a < b.
\end{equation}
Now we estimate the second product. 
\begin{equation}\label{EQ: PROD 2}
\begin{aligned}
        \log \left|\prod_{j\neq n} \left(1 - \frac{\mu_n}{\mu_j}\right)\right| &= \sum_{j < n}  \log \left( \frac{\mu_n}{\mu_j} - 1\right) + \sum_{j>n} \log \left(1 - \frac{\mu_n}{\mu_j}\right) \\
        &= \int_{\mu_1^{-}}^{\mu_{n-1}} \log \left(\frac{\mu_n}{x} - 1\right) d N_n(x) + \int_{\mu_{n+1}^{-}}^{\infty} \log \left( 1 - \frac{\mu_n}{x}\right) dN_n(x).
\end{aligned}
\end{equation}
where $N_n(x)$ is the counting function without the point at $x = \mu_n$ that
\begin{equation}
    N_n(x) = \begin{cases}
    N(x),\quad &x < \mu_n \\
    N(x) - 1, \quad &x \ge \mu_n.
    \end{cases}
\end{equation}
Therefore the previous estimate is modified to
\begin{equation}
    0\le  M^{-1/\beta} x^{1/\beta} + \delta(x) - N_n(x) \le 2
\end{equation}
and use integration by parts, the integrals of~\eqref{EQ: PROD 2} equals to
\begin{equation}\label{EQ: PROD 2 EST}
\begin{aligned}
    &\left[ N_n(x) \log \left(\frac{\mu_n}{x} - 1\right)\right]\bigg|_{\mu_1^{-}}^{\mu_{n-1}} + \left[N_n(x) \log\left(1 - \frac{\mu_n}{x}\right)\right]\bigg|_{\mu_{n+1}^{-}}^{\infty} \\
    & + \mu_n \int_{\mu_1}^{\mu_{n-1}} \frac{N_n(x) dx}{x(\mu_n-x)} + \mu_n \int_{\mu_{n+1}}^{\infty} \frac{N_n(x) dx}{x(\mu_n - x)} dx 
\end{aligned}
\end{equation}
For the boundary terms, it is simple to see 
\begin{equation}
    \lim_{x\to\infty} N_n(x) \log(1 - \frac{\mu_n}{x}) = 0,
\end{equation}
therefore the boundary terms become
\begin{equation}
\begin{aligned}
    &(n-1)\left[ \log \left(\frac{\mu_n}{\mu_{n-1}} - 1\right) - \log\left(1 - \frac{\mu_n}{\mu_{n+1}}\right)\right] \\&= (n-1)\left[\log\left(\frac{\mu_{n} - \mu_{n-1}}{\mu_{n+1} - \mu_n} \right)+ \log\left(\frac{\mu_{n+1}}{\mu_{n-1}}\right)\right].
\end{aligned}
\end{equation}
Using the assumption \textbf{A2}, 
\begin{equation}
\begin{aligned}
        \frac{\mu_n - \mu_{n-1}}{\mu_{n+1} - \mu_n} &= M \frac{n^{\beta} - (n-1)^{\beta} + o(n^{\beta - \sigma})}{\mu_{n+1} - \mu_n} \le \theta^{-1}M( \beta n^{s + \beta-1} + o(n^{s + \beta - \sigma}) ), \\
        \frac{\mu_{n+1}}{\mu_{n-1}} &= \frac{M (n+1)^{\beta}(1 + o((n+1)^{-\sigma}))}{M (n-1)^{\beta}(1 + o((n-1)^{-\sigma}))} = 1 + o(n^{-\sigma}).
\end{aligned}
\end{equation}
Therefore the boundary terms are bounded by 
\begin{equation}\label{EQ: BD}
    (n-1) ((s+\beta-\sigma)\log n + O(1)) = M^{-1/\beta} \frac{s+\beta-\sigma}{\beta}\mu_n^{1/\beta} \log \mu_n (1 + o(1)).
\end{equation}
 The integral terms in~\eqref{EQ: PROD 2 EST} can be estimated by
\begin{equation}
\begin{aligned}
    &\left| \mu_n  \int_{\mu_1}^{\mu_{n-1}} \frac{N_n(x) dx }{x(\mu_n - x)} -     \mu_n M^{-1/\beta}  \int_{\mu_1}^{\mu_{n-1}} \frac{x^{1/\beta} dx }{x(\mu_n - x)} \right| \le C \mu_n \int_{\mu_1}^{\mu_{n-1}} \frac{x^{(1-\sigma)/\beta}dx}{x(\mu_n - x)}  \\
    &\left| \mu_n \int_{\mu_{n+1}}^{\infty} \frac{N_n(x) dx}{x(\mu_n - x)} - \mu_n M^{-1/\beta}  \int_{\mu_{n+1}}^{\infty} \frac{x^{1/\beta} dx}{x(\mu_n - x)}  \right| \le C \mu_n \int_{\mu_{n+1}}^{\infty} \frac{x^{(1-\sigma)/\beta}dx}{x(x-\mu_n)} 
\end{aligned}
\end{equation}
The right-hand sides are bounded by 
\begin{equation}\label{EQ: REM 1}
\begin{aligned}
    C \mu_n \int_{\mu_1}^{\mu_{n-1}} \frac{x^{(1-\sigma)/\beta}dx}{x(\mu_n - x)} &= C \mu_n^{(1-\sigma)/\beta}\int_{\mu_1/\mu_n}^{\mu_{n-1}/\mu_n} \frac{dy}{y^{1-(1-\sigma)/\beta} (1-y)} \\&\le C_1 \mu_n^{(1-\sigma)/\beta}\left( \log \left(\frac{\mu_n}{\mu_n - \mu_{n-1}}\right) + O(1)\right) \\
    &\le C_1 \mu_n^{(1-\sigma)/\beta}\left( \log {\mu_n} + O(1)\right)
\end{aligned}
\end{equation}
and 
\begin{equation}\label{EQ: REM 2}
\begin{aligned}
    C \mu_n \int_{\mu_{n+1}}^{\infty} \frac{x^{(1-\sigma)/\beta}dx}{x(x- \mu_n)} &= C \mu_n^{(1-\sigma)/\beta}\int_{\mu_{n+1}/\mu_n}^{\infty} \frac{dy}{y^{1-(1-\sigma)/\beta} (y-1)} \\&=C_1 \mu_n^{(1-\sigma)/\beta}\left( \log\left(\frac{\mu_{n+1}-\mu_n}{\mu_n}\right) + O(1) \right) \\
    &=C_1\mu_n^{(1-\sigma)/\beta}  O(1).
\end{aligned}
\end{equation}
where $C_1$ is a positive constant independent of $n$. At last, we estimate the following integral
\begin{equation}\label{EQ: PV}
\begin{aligned}
        & \mu_n M^{-1/\beta}  \int_{\mu_1}^{\mu_{n-1}} \frac{x^{1/\beta} dx }{x(\mu_n - x)} +  \mu_n M^{-1/\beta}  \int_{\mu_{n+1}}^{\infty} \frac{x^{1/\beta} dx}{x(\mu_n - x)}  \\&=   \mu_n^{1/\beta} M^{-1/\beta}\left( \int_{\mu_1/\mu_n}^{\mu_{n-1}/\mu_n} \frac{dy}{y^{1-1/\beta}(1-y)} + \int_{\mu_{n+1}/\mu_n}^{\infty} \frac{dy}{y^{1-1/\beta}(1-y)} \right)
\end{aligned}
\end{equation}
In order to provide an estimate for~\eqref{EQ: PV}, we consider two cases. The first case is $1 - \frac{\mu_{n-1}}{\mu_n} \le \frac{\mu_{n+1}}{\mu_n} - 1$, then the integral 
\begin{equation}
\begin{aligned}
        &\int_{\mu_1/\mu_n}^{\mu_{n-1}/\mu_n} \frac{dy}{y^{1-1/\beta}(1-y)} + \int_{\mu_{n+1}/\mu_n}^{\infty} \frac{dy}{y^{1-1/\beta}(1-y)} \\
        &=     \int_{\mu_1 / \mu_n}^{\mu_{n-1}/\mu_n} \frac{dy}{y^{1-1/\beta}(1-y)} + \int_{2 -\mu_{n-1}/\mu_n }^{\infty} \frac{dy}{y^{1-1/\beta}(1-y)} - \int_{2 -\mu_{n-1}/\mu_n }^{\mu_{n+1}/\mu_n} \frac{dy}{y^{1-1/\beta}(1-y)} \\
        &=  \text{p.v.}\int_{0}^{\infty} \frac{dy}{y^{1-1/\beta}(1-y)} + o(1) - \int_{2 -\mu_{n-1}/\mu_n }^{\mu_{n+1}/\mu_n} \frac{dy}{y^{1-1/\beta}(1-y)} \\
        &\ge  \text{p.v.}\int_{0}^{\infty} \frac{dy}{y^{1-1/\beta}(1-y)} + o(1),
\end{aligned}
\end{equation}
 The second case is $1 - \frac{\mu_{n-1}}{\mu_n} > \frac{\mu_{n+1}}{\mu_n} - 1$, then 
\begin{equation}\label{EQ: SECOND}
    \begin{aligned}
  &\int_{\mu_1/\mu_n}^{\mu_{n-1}/\mu_n} \frac{dy}{y^{1-1/\beta}(1-y)} + \int_{\mu_{n+1}/\mu_n}^{\infty} \frac{dy}{y^{1-1/\beta}(1-y)} \\
        &=     \int_{\mu_1 / \mu_n}^{2 - \mu_{n+1}/\mu_n} \frac{dy}{y^{1-1/\beta}(1-y)} + \int_{\mu_{n+1}/\mu_n }^{\infty} \frac{dy}{y^{1-1/\beta}(1-y)} - \int^{2 -\mu_{n+1}/\mu_n }_{\mu_{n-1}/\mu_n} \frac{dy}{y^{1-1/\beta}(1-y)} \\  
        &\ge \text{p.v.}\int_{0}^{\infty} \frac{dy}{y^{1-1/\beta}(1-y)} + o(1) - C_2 \left|\log \frac{\mu_{n+1} - \mu_n }{\mu_{n} - \mu_{n-1}}\right| \\
        &\ge -C_3 (\log \mu_n  + C_4),
    \end{aligned}
\end{equation}
where $C_2, C_3, C_4$ are positive constants independent of $n$. Combine the previous estimates~\eqref{EQ: BD},~\eqref{EQ: REM 1},~\eqref{EQ: REM 2} and~\eqref{EQ: SECOND}, we obtain the following inequality
\begin{equation}
\begin{aligned}
  \log \left|\prod_{j\neq n} \left(1 - \frac{\mu_n}{\mu_j}\right)\right| 
 \ge -K_0 \mu_n^{1/\beta} M^{-1/\beta} (\log \mu_n + K_1 )
\end{aligned}
\end{equation}
for certain $K_0, K_1 > 0$.
\end{proof}

\bibliographystyle{plain}
\bibliography{main}

\begin{thebibliography}{10}

\bibitem{agmon1968asymptotic}
Shmuel Agmon.
\newblock Asymptotic formulas with remainder estimates for eigenvalues of
  elliptic operators.
\newblock {\em Archive for Rational Mechanics and Analysis}, 28(3):165--183,
  1968.

\bibitem{beals1970asymptotic}
Richard Beals.
\newblock Asymptotic behavior of the green's function and spectral function of
  an elliptic operator.
\newblock {\em Journal of Functional Analysis}, 5(3):484--503, 1970.

\bibitem{fattorini1971exact}
Hector~O Fattorini and David~L Russell.
\newblock Exact controllability theorems for linear parabolic equations in one
  space dimension.
\newblock {\em Archive for Rational Mechanics and Analysis}, 43(4):272--292,
  1971.

\bibitem{he2022much}
Yuchen He, Hongkai Zhao, and Yimin Zhong.
\newblock How much can one learn a partial differential equation from its
  solution?
\newblock {\em arXiv preprint arXiv:2204.04602}, 2022.

\bibitem{hormander1968riesz}
Lars~Valter H{\"o}rmander.
\newblock On the riesz means of spectral functions and eigenfunction expansions
  for elliptic differential operators.
\newblock {\em Matematika}, 12(5):91--130, 1968.

\bibitem{logunov2018quantitative}
Alexander Logunov and Eugenia Malinnikova.
\newblock Quantitative propagation of smallness for solutions of elliptic
  equations.
\newblock In {\em Proceedings of the International Congress of Mathematicians
  (ICM 2018) (In 4 Volumes) Proceedings of the International Congress of
  Mathematicians 2018}, pages 2391--2411. World Scientific, 2018.

\bibitem{berezanskiui1968expansions}
Berezansky~Yurij Makarovich.
\newblock {\em Expansions in eigenfunctions of selfadjoint operators},
  volume~17.
\newblock American Mathematical Soc., 1968.

\bibitem{safarov1997asymptotic}
Yu~Safarov and D~Vassilev.
\newblock {\em The asymptotic distribution of eigenvalues of partial
  differential operators}, volume 155.
\newblock American Mathematical Soc., 1997.

\bibitem{schwartz1943etude}
Laurent Schwartz.
\newblock {\em {\'E}tude des sommes d'exponentielles r{\'e}elles}.
\newblock Hermann Paris, 1943.

\bibitem{stewart1998perturbation}
Gilbert~W Stewart.
\newblock Perturbation theory for the singular value decomposition.
\newblock Technical report, 1998.

\bibitem{vu2011singular}
Van Vu.
\newblock Singular vectors under random perturbation.
\newblock {\em Random Structures \& Algorithms}, 39(4):526--538, 2011.

\bibitem{wang2015singular}
Rongrong Wang.
\newblock Singular vector perturbation under gaussian noise.
\newblock {\em SIAM Journal on Matrix Analysis and Applications},
  36(1):158--177, 2015.

\bibitem{wedin1972perturbation}
Per-{\AA}ke Wedin.
\newblock Perturbation bounds in connection with singular value decomposition.
\newblock {\em BIT Numerical Mathematics}, 12(1):99--111, 1972.

\bibitem{widder2015laplace}
David~Vernon Widder.
\newblock Laplace transform (pms-6).
\newblock In {\em Laplace Transform (PMS-6)}. Princeton university press, 2015.

\bibitem{zielinski1998asymptotic}
Lech Zielinski.
\newblock Asymptotic distribution of eigenvalues for some elliptic operators
  with simple remainder estimates.
\newblock {\em Journal of Operator Theory}, pages 249--282, 1998.

\end{thebibliography}
\end{document}